\documentclass[reqno]{amsart}
\usepackage{amsmath,amsthm,amssymb,bbm,color,verbatim,tikz,mathabx}
\usetikzlibrary{matrix}

\input xy 
\xyoption{all}

\numberwithin{equation}{section}

\renewcommand{\marginpar}[2][]{}

\usepackage{hyperref}
\hypersetup{colorlinks=true, linkcolor=blue, citecolor=black}

\newcommand{\ORD}{\mathop{{\rm ORD}}}

\renewcommand{\P}{{\mathbb P}}
\newcommand{\E}{{\mathbb E}}

\newcommand{\Q}{{\mathbb Q}}

\newcommand{\Sacks}{\mathop{\rm Sacks}}
\newcommand{\Add}{\mathop{\rm Add}}

\renewcommand{\Col}{\mathop{\rm Col}}
\newcommand{\Ult}{\mathop{\rm Ult}}
\newcommand{\forces}{\Vdash}

\newcommand{\restrict}{\upharpoonright}

\renewcommand{\>}{\rangle}
\newcommand{\elemsub}{\prec}

\newcommand{\st}{\mid}

\newcommand{\supp}{\mathop{\rm supp}}

\newcommand{\dom}{\mathop{\rm dom}}
\newcommand{\ran}{\mathop{\rm ran}}

\newcommand{\ro}{\mathop{\rm ro}}

\newcommand{\crit}{\mathop{\rm crit}}
\renewcommand{\span}{\mathop{\rm span}}

\renewcommand{\and}{\mathop{\&}}

\newtheorem{theorem}{Theorem}
\newtheorem{lemma}{Lemma}
\newtheorem{corollary}[theorem]{Corollary}
\newtheorem{claim}{Claim}

\theoremstyle{definition}
\newtheorem{question}{Question}
\newtheorem{remark}{Remark}

\newtheorem{definition}{Definition}

\usepackage{stmaryrd} 

\subjclass[2000]{03E35, 03E55}
\date{\today}

\begin{document}

\title{Indestructibility of generically strong cardinals}

\author[Brent Cody]{Brent Cody}
\address[Brent Cody]{ 
Virginia Commonwealth University,
Department of Mathematics and Applied Mathematics,
1015 Floyd Avenue, PO Box 842014, Richmond, Virginia 23284
} 
\email[B. ~Cody]{bmcody@vcu.edu} 
\urladdr{http://www.people.vcu.edu/~bmcody/}

\author[Sean Cox]{Sean Cox}
\address[Sean Cox]{
Virginia Commonwealth University,
Department of Mathematics and Applied Mathematics,
1015 Floyd Avenue, PO Box 842014, Richmond, Virginia 23284
} 
\email[S. ~Cox]{scox9@vcu.edu} 
\urladdr{http://skolemhull.wordpress.com/}

\begin{abstract}
Foreman~\cite{Foreman:CalculatingQuotientAlgebras} proved the \emph{Duality Theorem}, which gives an algebraic characterization of certain ideal quotients in generic extensions.  As an application he proved that generic supercompactness of $\omega_1$ is preserved by any proper forcing.  We generalize portions of Foreman's Duality Theorem to the context of generic extender embeddings and \emph{ideal extenders} (as introduced by Claverie~\cite{Claverie:IdealsIdealExtendersAndForcingAxioms}).  As an application we prove that if $\omega_1$ is generically strong, then it remains so after adding any number of Cohen subsets of $\omega_1$; however many other $\omega_1$-closed posets---such as $\text{Col}(\omega_1, \omega_2)$---can destroy the generic strongness of $\omega_1$.  This generalizes some results of Gitik-Shelah~\cite{GitikShelah:OnCertainIndestructibility} about indestructibility of strong cardinals to the generically strong context.  We also prove similar theorems for successor cardinals larger than $\omega_1$.

\end{abstract}

\subjclass[2010]{Primary 03E35; Secondary 03E55}

\keywords{generic large cardinals, strong cardinal, precipitous, indestructibility}

\maketitle

\section{Introduction}\label{sectionintroduction}

Many conventional large cardinal properties are witnessed by the existence of ultrafilters or elementary embeddings. Some of these properties can be viewed as special cases of more general concepts, namely generic large cardinal properties \cite[Remark 3.14]{Foreman:Handbook}, which are usually witnessed by \emph{generic} ultrafilters and \emph{generic} elementary embeddings, and which can be arranged to hold at small cardinals. For example, by collapsing below a measurable cardinal, one can arrange that there is a precipitous ideal $I$ on $\omega_1$ witnessing the generic measurability of $\omega_1$ in the sense that forcing with $\wp(\kappa)/I-\{[0]\}$ naturally yields a $V$-normal $V$-ultrafilter on $\wp(\kappa)^V$ with a well-founded ultrapower in the extension (see \cite[Section 17.1]{Cummings:Handbook}).



Preserving large cardinals through forcing by lifting elementary embeddings has been a major theme in set theory \cite{Cummings:Handbook}. In some cases, large cardinals can be made indestructible by wide classes of forcing. For example, given a supercompact cardinal $\kappa$, Laver \cite{Laver:MakingSupercompactnessIndestructible} used a preparatory forcing iteration to produce a model in which $\kappa$ remains supercompact, and the supercompactness of $\kappa$ is indestructible by any further $\kappa$-directed closed forcing. Gitik and Shelah \cite{GitikShelah:OnCertainIndestructibility} proved that the strongness of a cardinal $\kappa$ can be made indestructible by $\kappa^+$-weakly closed forcing satisfying the Prikry condition. As shown by Hamkins \cite{Hamkins:TheLotteryPreparation}, one can obtain indestructibility results for many large cardinals by using a preparatory forcing he calls the lottery preparation. 


It is natural to wonder whether similar results hold for some generic large cardinal properties.  Let us first clear up some possibly confusing terminology in the literature, which will also help the reader more easily understand the division of Sections \ref{sec_Surgery} and \ref{sec_Duality} in this paper and Claverie's~\cite{Claverie:IdealsIdealExtendersAndForcingAxioms} distinction between \emph{ideally strong} and \emph{generically strong}.  Conventional large cardinals---such as measurable and supercompact cardinals---admit equivalent definitions in terms of either embeddings or ultrafilters.  This is not true of generic large cardinals.  Let us say that a cardinal $\kappa$ is:
\begin{itemize}
 \item \emph{generically measurable} iff there is a poset which forces that there is an elementary embedding $j:V \to M$ with $M$ wellfounded and $\text{crit}(j) = \kappa$;
 \item \emph{ideally measurable} iff there is a precipitous ideal on $\kappa$; equivalently, $\kappa$ is generically measurable and the poset which witnesses the generic measurability is of the form $\wp(\kappa)/I$ for some ideal $I$ on $\kappa$.
  \item \emph{generically supercompact} iff for every $\lambda > \kappa$ there is a poset which forces that there is an elementary embedding $j:V \to M$ with $M$ wellfounded, $\text{crit}(j) = \kappa$, and $j " \lambda \in M$.
  \item \emph{ideally supercompact} iff $\kappa$ is generically supercompact and for each $\lambda$ the witnessing poset is of the form $\wp(P_\kappa(\lambda))/I$ for some  some normal precipitous ideal $I$ on $P_\kappa(\lambda)$.  This is what Foreman~\cite{Foreman:CalculatingQuotientAlgebras} calls generic supercompactness.
\end{itemize}

Kakuda \cite{Kakuda:OnACondtionForCohenExtensionsWhichPreservePrecipitousIdeals} and Magidor \cite{Magidor:PrecipitousIdealsAndSigma14Sets} independently proved that if $\kappa$ is a regular uncountable cardinal and $\P$ is a poset satisfying the $\kappa$-chain condition, then $I\subseteq \wp(\kappa)$ is a precipitous ideal if and only if $\P$ forces that the ideal on $\kappa$ generated by $I$ is precipitous. In particular, if $\kappa$ is ideally measurable in the sense given above, then it remains so in any $\kappa$-c.c. forcing extension.  Their proof also shows that generic measurability of $\kappa$ is preserved by any $\kappa$-cc forcing; this is an easier result.   

Foreman~\cite{Foreman:CalculatingQuotientAlgebras} proved a \emph{Duality Theorem} which gives an algebraic characterization of ideal quotients arising in many settings where generic large cardinal embeddings are lifted to generic extensions.  As an application, he proved the following theorem:
\begin{theorem}[Foreman~\cite{Foreman:CalculatingQuotientAlgebras}]\label{thm_Foreman}
If $\omega_1$ is generically supercompact or ideally supercompact, then it remains so after any proper forcing. 
\end{theorem}
It should be noted that Foreman's Theorem \ref{thm_Foreman}, and also our Theorem \ref{maintheorem} below, require no ``preparation forcing".  We caution the reader that what Foreman calls generic supercompactness is what we're calling ideal supercompactness.  We note that Foreman's Duality Theorem is used in his proof of preservation of ideal supercompactness; but the Duality Theorem is \emph{not} used in the proof of preservation of generic supercompactness.

Claverie \cite{Claverie:IdealsIdealExtendersAndForcingAxioms} introduced the notions of generically strong and  ideally strong cardinals.  These are defined in Section \ref{sec_Preliminaries} below, but roughly:  $\kappa$ is generically strong iff there are arbitrarily strong wellfounded generic embeddings of $V$; and $\kappa$ is ideally strong iff its generic strongness is witnessed by what Claverie calls \emph{ideal extenders}.  In summary, just as precipitous ideals can be used to witness the generic measurability or supercompactness of a cardinal, precipitous ideal extenders can be used to witness generic strongness.

In this article we prove the following indestructibility result concerning ideally strong cardinals and generically strong cardinals:
\begin{theorem}\label{maintheorem}
If $\omega_1$ is ideally strong or generically strong, then it remains so after forcing to add any number of Cohen subsets of $\omega_1$.
\end{theorem}
\noindent Let us remark that Theorem \ref{maintheorem} is generalized below (see Theorem \ref{maintheorem_generalized} on page \pageref{maintheorem_generalized}) to include the case in which $\kappa>\omega_1$, assuming other technical requirements related to internal approachability.


Theorem \ref{maintheorem} will follow from Theorem \ref{theoremsurgery} and Theorem \ref{thm_Duality_Extenders} below.  Theorem \ref{theoremsurgery}---which applies to any generically strong setting---involves lifting a generic strongness embedding through Cohen forcing using a `surgery' argument (see \cite[Theorem 25.1]{Cummings:Handbook} and \cite{CodyMagidor}) rather than a master condition argument which is possible in the supercompactness context.  Theorem \ref{thm_Duality_Extenders}---which is specific to the ideally strong setting---extends portions of Foreman's duality theory (see \cite{Foreman:CalculatingQuotientAlgebras}
and \cite{Foreman:Handbook}) into the context of ideal extenders, and allows us to conclude that a certain $(\kappa,\lambda)$-ideal extender derived from a lifted embedding is precipitous.  Theorems \ref{theoremsurgery} and \ref{thm_Duality_Extenders} are then used to prove Theorem \ref{maintheorem} in Section \ref{sectionproof}.

It is natural to ask whether Theorem \ref{maintheorem} can be strengthened; for example, whether the generic or ideal strongness of $\omega_1$ is necessarily preserved by \emph{all} $\omega_1$-closed forcings (not just forcings to add Cohen subsets to $\omega_1$).  The answer is no, for the following reason, as alluded to in Gitik-Shelah~\cite{GitikShelah:OnCertainIndestructibility} in the conventional large cardinal setting.  Suppose $G$ is $\text{Col}(\kappa, \kappa^+)$-generic over $V$, and that in $V[G]$, the cardinal $\kappa$ is generically measurable.  Then $V$ must have an inner model with a Woodin cardinal.  Otherwise, the core model $K$ exists in $V$ and is absolute to $V[G]$; and moreover (see \cite{MR2963017}) $V[G]$ believes that $K$ computes $\kappa^+$ correctly; but this is a contradiction, because $\kappa^{+K} \le \kappa^{+V} < \kappa^{+V[G]}$.

\section{Preliminaries}\label{sec_Preliminaries}

\subsection{Ideal Extenders}

We refer the reader to Kanamori~\cite{Kanamori:Book} for the definition of an extender.

Claverie~\cite{Claverie:IdealsIdealExtendersAndForcingAxioms} gives a detailed account of \emph{ideal extenders} and when ideal extenders are precipitous.  We recount his main definitions in this section. Claverie defines a set $F$ to be a \emph{$(\kappa,\lambda)$-system of filters} if $F\subseteq\{(a,x)\in[\lambda]^{<\omega}\times \wp([\kappa]^{<\omega}): x\subseteq[\kappa]^{|a|}\}$ and for each $a\in[\lambda]^{<\omega}$ the set $F_a:=\{x:(a,x)\in F\}$ is a filter on $[\kappa]^{|a|}$. The \emph{support} of $F$ is $\supp(F)=\{a\in[\lambda]^{<\omega}:F_a\neq\{[\kappa]^{|a|}\}\}$. In what follows we will identify each finite subset of $\lambda$ with the corresponding unique increasing enumeration of its elements; in other words, $a\in[\lambda]^{<\omega}$ will be identified with the function $s_a:|a|\to a$ listing the elements of $a$ in increasing order. Given $a,b\in[\lambda]^{<\omega}$ with $a\subseteq b$, let $s:|a|\to|b|$ be the unique function such that $a(n)=b(s(n))$. For $x\in \wp([\kappa]^{|a|})$, define $x_{a,b}=\{\<u_i:i<|b|\>\in[\kappa]^{|b|}\st \<u_{s(j)}:j<|a|\>\in x\}$. A $(\kappa,\lambda)$-system of filters $F$ is called \emph{compatible} if for all $a\subseteq b$ with $a,b\in\supp(F)$ one has
$x\in F_a$ if and only if $x_{a,b}\in F_b$. Given $a\in[\lambda]^{<\omega}$ and $x\in F_a^+$, we define $F'=\span\{F,(a,x)\}$, \emph{the span of $F$ and $(a,x)$}, to be the smallest $(\kappa,\lambda)$-system of filters such that $F\subseteq F'$ and $(a,x)\in F'$.

If $F$ is a single filter on $\kappa$, the usual Boolean algebra $\mathbb{B}_F:=\wp(\kappa)/F$ is forcing equivalent to the poset $\mathbb{E}_F$ whose conditions are filters of the form $\overline{F \cup \{ S \}}$, where $S \in F^+$;\footnote{Here $\overline{F \cup \{ S \} }$ denotes the filter generated by $F \cup \{S \}$.} the poset $\mathbb{E}_F$ is ordered by reverse inclusion.  The following definition generalizes the latter poset to \emph{systems} of filters:
\begin{definition}\cite[Definition 4.3(v)]{Claverie:IdealsIdealExtendersAndForcingAxioms}\label{def_AssocForcing}
Given a $(\kappa,\lambda)$-system of filters $F$, the \emph{forcing associated to $F$} is denoted by $\E_F$ and consists of all conditions $p=F^p$ where $F^p$ is a compatible $(\kappa,\lambda)$-system of filters such that $\supp(p)=\supp(F^p)\subseteq \supp(F)$ and $F^p$ is generated by one point $x\in F_a^+$ for some $a\in \supp(p)$, i.e. $F^p=\span\{F,(a,x)\}$. The ordering on $\E_F$ is defined by $p\leq q$ if and only if $\supp(q)\subseteq\supp(p)$ and for all $a\in\supp(q)$, $F^q_a\subseteq F^p_a$, in other words $F^q\subseteq F^p$.
\end{definition}

If $F$ is a compatible $(\kappa,\lambda)$-system of filters and $\dot{G}$ is the $\mathbb{E}_F$-name for the generic object, let $\dot{E}_F$ be the name for $\bigcup \dot{G}$. Then clearly $\dot{E}_F$ is forced to be a $(\kappa,\lambda)$-system of filters, and indeed, as shown in \cite{Claverie:IdealsIdealExtendersAndForcingAxioms}, each $(\dot{E}_F^G)_a=(\bigcup G)_a$ is a $V$-ultrafilter. Claverie defines two additional combinatorial properties of $(\kappa,\lambda)$-systems of filters, \emph{potential normality} and \emph{precipitousness} (see \cite[Definition 4.4]{Claverie:IdealsIdealExtendersAndForcingAxioms}) such that the following two facts hold. First, if $F$ is a compatible potentially normal $(\kappa,\lambda)$-system of filters, then $\dot{E}_F^G$ is a $(\kappa,\lambda)$-extender over $V$; secondly, if $F$ is in addition precipitous, then the generic ultrapower $j_{\dot{E}_F^G}:V\to \Ult(V,\dot{E}_F^G)$ is well-founded.

\begin{definition}\cite[Definition 4.5]{Claverie:IdealsIdealExtendersAndForcingAxioms}\label{def_IdealExtender}
Let $\kappa<\lambda$ be ordinals. $F$ is a \emph{$(\kappa,\lambda)$-ideal extender} if it is a compatible and potentially normal $(\kappa,\lambda)$-system of filters such that for each $a\in\supp(F)$, the filter $F_a$ is ${<}\kappa$-closed.
\end{definition}
\noindent As noted in the introduction, precipitous $(\kappa,\lambda)$-ideal extenders can be used to witness generic strongness, in a way similar to that in which precipitous ideals can be used to witness generic measurability or generic supercompactness:

\begin{definition}\label{def_ideally_strong}\cite[Definition 4.9]{Claverie:IdealsIdealExtendersAndForcingAxioms} A regular cardinal $\kappa$ is called \emph{ideally strong} if and only if for all $A\subseteq\ORD$, $A\in V$, there is some precipitous $(\kappa,\lambda)$-ideal extender $F$ such that whenever $G$ is $\E_F$-generic, and $\dot{E}_F^G$ is the corresponding $V$-$(\kappa,\lambda)$-extender, one has $A\in\Ult(V,\dot{E}^G_F)$. 
\end{definition}

Ideal strongness is a specific form of generic strongness, which is defined as follows:\footnote{Definition \ref{def_generically_strong} is equivalent to a first-order definition; in particular we could just require the domain of $j$ to be $H^V_{\theta_{\mathbb{P},A}}$ where $\theta_{\mathbb{P},A} \ge |\mathbb{P}|^{+V}$ and $A \in H_{\theta_{\mathbb{P},A}}$.  Then $\theta_{\mathbb{P},A}$ will always be uncountable in $V^{\mathbb{P}}$ and by well-known arguments such embeddings lift to domain $V$.}
\begin{definition}\label{def_generically_strong}\cite[Definition 4.20]{Claverie:IdealsIdealExtendersAndForcingAxioms}
A cardinal $\kappa$ is \emph{generically strong} if for every set of ordinals $A$, in $V$, there is a poset $\mathbb{P}$ such that in $V^{\mathbb{P}}$ there is a definable embedding $j: V \to M$ with $\text{crit}(j) = \kappa$, $M$ wellfounded, and $A \in M$.
\end{definition}

Let us briefly mention a characterization of ideally strong cardinals which we will use in our proof of Theorem \ref{maintheorem}.

\begin{lemma}\label{lem_ideally_strong_characterization}
A regular cardinal $\kappa$ is ideally strong (as in Definition \ref{def_ideally_strong}) if and only if for every $\lambda\geq\kappa$ there is a precipitous $(\kappa,\lambda)$-ideal extender $F$ such that whenever $G$ is $(V,\E_F)$-generic the generic ultrapower $j_{G}:V\to M=\Ult(V,\dot{E}^G_F)$ satisfies the following properties. 
\begin{enumerate}
\item $\crit(j_{G})=\kappa$
\item $j_{G}(\kappa)>\lambda$
\item $H_\lambda^V\subseteq M$
\end{enumerate}
\end{lemma}

One can also use the $V_\alpha$-hierarchy to characterize ideally strong cardinals. However, to obtain $V_\alpha\subseteq M=\Ult(V,\dot{E}_F^G)$, the length of the generic extender $\dot{E}_F^G$ must be at least $(|V_\alpha|^+)^V$ (see \cite[Exercise 26.7]{Kanamori:Book}). For notational convenience, in what follows, we will use the characterization of ideally strong cardinals given in Lemma \ref{lem_ideally_strong_characterization}.

Just as one can produce precipitous ideals by collapsing below large cardinals, Claverie proves \cite[Lemma 4.8]{Claverie:IdealsIdealExtendersAndForcingAxioms} that if $\kappa$ is $\lambda$-strong in $V$ witnessed by a $(\kappa,\lambda)$-extender $E$, and $\mu<\kappa$ is a cardinal, then forcing with the Levy collapse $\Col(\mu,{<}\kappa)$ produces a model in which $\mu^+=\kappa$ and \[F=\{(a,x):\textrm{$x\subseteq[\kappa]^{|a|}$ and $\exists y$ such that $(a,y)\in E$ and $y\subseteq x$}\}\] is a precipitous $(\kappa,\lambda)$-ideal extender.

Notice that if $F$ is a filter, $F_0$ and $F_1$ are filter extensions of $F$, and $F_1 \nsubseteq F_0$ as witnessed by some $S \in F_1 - F_0$, then the filter $F'_0$ generated by $F_0 \cup \{ S^c \}$ extends $F_0$ and is incompatible with $F_1$, in the sense that there is no proper filter extension of $F'_0 \cup F_1$.  The following lemma is a generalization of this fact, and will be used in our proof of Theorem \ref{maintheorem}.
\begin{lemma}\label{lem_E_F_separative}
Given a $(\kappa,\lambda)$-system of filters $F$, the associated forcing $\mathbb{E}_F$ is a separative poset.
\end{lemma}

\begin{proof}
Suppose $p,q\in\mathbb{E}_F$ are conditions and that $p\nleq q$. We have $p=\span\{F,(a,x)\}$ and $q=\span\{F,(b,y)\}$ for some $a,b\in[\lambda]^{<\omega}$ and for some $x\in (F_a)^+$ and $y\in (F_b)^+$. Since $p\nleq q$ we have that $F^q_c\not\subseteq F^p_c$ for some $c\in\supp(q)$. Let $z\in F^q_c\setminus F^p_c$, then it follows that $[\kappa]^{|c|}\setminus z\in (F^p_c)^+\subseteq F_c^+$. Since $F^p$ is a compatible $(\kappa,\lambda)$-system of filters and $x\in F^p_a$, we have $x_{a,a\cup c}\in F^p_{a\cup c}$ and similarly $([\kappa]^{|c|}\setminus z)_{c,a\cup c}\in (F^p_{a\cup c})^+$. Thus, $x_{a,a\cup c}\cap ([\kappa]^{|c|}\setminus z)_{c, a\cup c}\in (F^p_{a\cup c})^+\subseteq F_{a\cup c}^+$, which means that $r=\span\{F,(a\cup c,x_{a,a\cup c}\cap ([\kappa]^{|c|}\setminus z)_{c, a\cup c})\}$ is a condition in $\E_F$. Since $x_{a,a\cup c}\in F^r_{a\cup c}$ and thus $x\in F^r_a$, it follows that $r\leq p$. Furthermore, since $([\kappa]^{|c|}\setminus z)_{c,a\cup c}\in F^r_{a\cup c}$ we have $[\kappa]^{|c|}\setminus z\in F^r_c$, and thus $r\perp q$. Therefore $\E_F$ is separative.
\end{proof}

\subsection{Internal Approachability}

In this section we discuss the key technical issue which comes up in generalizing Theorem \ref{maintheorem} to the case in which $\kappa>\omega_1$.

The following definition is standard; see Foreman-Magidor~\cite{ForemanMagidor:LargeCardinalsandDefinable} for details:
\begin{definition}\label{defn_mu_internally_approachable}
Let $\mu$ be a regular cardinal.  A set $X$ with $\mu\subseteq X$ is called \emph{$\mu$-internally approachable} iff there is a $\subseteq$-increasing sequence $\vec{N} = \langle N_i \ : \ i < \mu \rangle$ such that:
\begin{itemize}
 \item $\vec{N}$ is $\subseteq$-continuous (i.e., $N_j=\bigcup_{i<j}N_i$ for each limit ordinal $j<\mu$);
 \item $|N_i| < \mu$ for each $i < \mu$;
 \item $X = \bigcup_{i < \mu} N_i$;
 \item $\langle N_i \ : \ i < j \rangle \in X$ for all $j < \mu$
\end{itemize}
We let $\text{IA}_\mu$ denote the class of $\mu$-internally approachable sets.  
\end{definition}

\begin{remark}
If $\vec{N}=\<N_i\st i<\mu\>$ witnesses that $X\in\text{IA}_\mu$, then $N_i\cap\mu\in\mu$ for club-many $i<\mu$. This is because $\mu$ is an uncountable regular cardinal and the assumptions on $\vec{N}$ imply that $\{N_i\st i<\mu\}$ is a closed unbounded subset of $P_\mu(X)$. It is then an easy abstract fact that every club subset of $P_\mu(X)$ contains a club subset of $z$ such that $z\cap\mu\in \mu$.
\end{remark}

Note that almost every countable subset of $H_\theta$ (for $\theta$ regular uncountable) is $\omega$-internally approachable; this is just because models of set theory are closed under finite sequences.  A simple closing-off argument shows that for any regular $\mu \le \theta$, the set $\text{IA}_\mu \cap \wp_{\mu^+}(H_\theta)$ is always stationary.  Note also that for any $X \in \wp_{\mu^+}(H_\theta)$: $X \in \text{IA}_\mu$ iff $H_X \in \text{IA}_\mu$, where $H_X$ is the transitive collapse of $X$.  We refer the reader to Foreman-Magidor~\cite{ForemanMagidor:LargeCardinalsandDefinable} for proofs of the following lemma.
\begin{lemma}\label{lem_GCH_equiv}
Suppose GCH holds.  Let $\theta \ge \mu$ be regular cardinals.  Then the following sets are equal modulo the nonstationary ideal on $\wp_{\mu^+}(H_\theta)$.
\begin{itemize}
 \item $\text{IA}_\mu \cap \wp_{\mu^+}(H_\theta)$;
 \item The set of $N \in \wp_{\mu^+}(H_\theta)$ such that ${}^{<\mu} N \subset N$. 
\end{itemize}
\end{lemma}

If $H$ is a countable transitive model and $\mathbb{P} \in H$ is a poset, then one can use a diagonalization argument to build $(H,\mathbb{P})$-generic filters.  If $H$ is uncountable then the diagonalization argument still works, provided that $H$ is internally approachable and believes $\mathbb{P}$ is sufficiently closed and well-orderable:

\begin{lemma}\label{lem_IA_gives_generics}
Suppose $H$ is a transitive $ZF^-$ model, $\mu>\omega$ is a cardinal, $H \in \text{IA}_\mu$, $\Q \in H$ is a poset which is ${<}\mu$-closed in $H$, and for every $N \in H$ if $|N|^V < \mu$ then $|N|^H < \mu$.  Suppose $H$ has some wellorder of $\Q$.  Then for every $p \in \Q$ there exists a $g \subset \Q$ such that $p \in g$ and $g$ is $(H, \Q)$-generic.
\end{lemma}
\begin{proof}
Fix some $p \in \Q$.  Let $\Delta \in H$ be a wellorder of $\Q$, and let $\langle N_i \ : \ i < \mu \rangle$ witness that $H \in \text{IA}_{\mu}$.  Note that without loss of generality we can assume that:
\begin{enumerate}
 \item $p \in N_0$ and $\Delta \in N_0$;
 \item for all $\ell < \mu$, $\langle N_i \ : \ i \le \ell \rangle \in N_{\ell+1}$
 \item $N_i \prec H$ for all $i < \mu$
\end{enumerate}
For a set $N$ and a condition $r$, let us call $r$ a \emph{weakly $N$-generic condition} iff for every $D \in N$ which is dense, there is some $s \ge r$ such that $s \in D$; note we do \textbf{not} require that $s \in D \cap N$, so the upward closure of $r$ might fail to be an $(N, \mathbb{Q})$-generic filter in the usual sense.
 
 Recursively construct a descending sequence $\langle p_i \ : \ i < \mu \rangle$ as follows:
\begin{itemize}
 \item $p_0:=p$;
 \item Assuming $p_i$ has been defined and is an element of $N_{i+1}$, let $p_{i+1}$ be the $\Delta$-least weakly $N_{i+1}$-generic condition below $p_i$. Such a condition exists because $|\{D\in N_{i+1}:\textrm{$D\subseteq\Q$ is dense below } p_i \}|^H \le |N_{i+1}|^H < \mu$ and $\Q$ is ${<}\mu$-closed in $H$.
 \item Assuming $j < \mu$ is a limit ordinal and $\langle p_i \ : \ i < j \rangle$ has been defined, let $p_j$ be the $\Delta$-least lower bound of $\langle p_i \ : \ i < j \rangle$ in $\Q$.  Such a condition exists because  $\<p_i\st i<j\>\in H$ follows from $\<N_i\st i<j\>\in H$ (and $H \models$ $\mathbb{Q}$ is ${<}\mu$-closed).
\end{itemize}
A straightforward inductive proof shows that $p_j \in N_{j+1}$ for all $j$, that $\langle p_i \ : \ i < \ell \rangle \in N_{\ell+1}$ for all limit ordinals $\ell$, and that whenever $j$ is a successor ordinal then $p_j$ is a weakly generic condition for $N_j$.  It follows (since $H = \bigcup_{i < \mu} N_i$) that $\{ p_i \ : \ i < \mu \}$ generates an $H$-generic filter. 
\end{proof}

The role of $\text{IA}_\mu$ in the theory of generic embeddings (with critical point $\ge \omega_2$) has been explored in detail in Foreman-Magidor~ \cite{ForemanMagidor:LargeCardinalsandDefinable}. It also plays a role in the current paper.

\begin{definition}\label{def_GenericStrongIA}
Let $\mu$ be a regular cardinal and suppose $\kappa=\mu^+$.  We say that $\kappa$ is \emph{generically strong on $\text{IA}_\mu$} if and only if for every $\lambda > \kappa$ and $A \in V$ there is a poset $\P$ such that in $V^\P$ there is a definable embedding $j:V \to M$ witnessing that $\kappa$ is generically strong for $A$ as in Definition \ref{def_generically_strong}, which in addition has the following properties\footnote{The requirement that $M^{<\mu}\subseteq M$ is somewhat ad hoc and unnecessary for the results of this article when $\kappa=\omega_1$ and $\mu=\omega$; see Remark \ref{rem_RedundantAtOmega1}. We suspect that the main results of this article for $\kappa>\omega_1$ can be proven without using this ad hoc requirement. However, currently we do not see how to eliminate its use in the proof of Claim \ref{claim_Mgeneric} of Theorem \ref{theoremsurgery} below.} in $V^{\mathbb{P}}$:  
\begin{enumerate}
 \item\label{item_ClosedLessMu} $M^{<\mu}\subseteq M$ and
 \item\label{item_IA} $M \models {H}^V_{\lambda}\in \text{IA}_\mu$.
\end{enumerate}


We say that $\kappa$ is \emph{ideally strong on $\text{IA}_\mu$} if and only if for every cardinal $\lambda\geq\kappa$ there is a precipitous $(\kappa,\lambda^+)$-ideal extender $F$, whose associated generic ultrapower $j:V\to M$ satisfies requirements (\ref{item_ClosedLessMu}) and (\ref{item_IA}).
\end{definition}

\begin{remark}
Of course, if $\kappa$ is ideally strong on $\text{IA}_\mu$ then $\kappa$ is generically strong on $\text{IA}_\mu$. 
\end{remark}
\begin{remark}\label{rem_RedundantAtOmega1}
Notice that if $\kappa=\omega_1$ and $\mu = \omega$, then ``$\kappa$ is generically strong on $\text{IA}_\mu$" is equivalent to ``$\kappa$ is generically strong", and ``$\kappa$ is ideally strong on $\text{IA}_\mu$" is equivalent to ``$\kappa$ is ideally strong".  This is simply because all the relevant models are closed under \emph{finite} (i.e.\ $<\mu$) sequences.
\end{remark}


Now we are ready to state a theorem generalizing Theorem \ref{maintheorem} to include the case $\kappa>\omega_1$.  

\begin{theorem}\label{maintheorem_generalized} 
Suppose $\mu$ is a regular cardinal, $\kappa=\mu^+$, and $2^{<\kappa}=\kappa$. If $\kappa$ is generically strong on $\text{IA}_\mu$ then this is preserved by adding any number of Cohen subsets of $\kappa$.  Moreover, if $\kappa$ is ideally strong on $\text{IA}_\mu$, then this is preserved by adding any number of Cohen subsets of $\kappa$. 
\end{theorem}

The remainder of the paper is a proof of Theorem \ref{maintheorem_generalized}.  Note that by Remark \ref{rem_RedundantAtOmega1}, Theorem \ref{maintheorem_generalized} implies Theorem \ref{maintheorem}.


\section{Lifting generic embeddings using surgery}\label{sec_Surgery}

\begin{theorem}\label{theoremsurgery}
Suppose that $\mu\geq\omega$ is a regular cardinal, $\kappa=\mu^+$, and $2^{<\kappa}=\kappa$. Let $\theta$ and $\eta$ be cardinals with $\eta>\theta^+$, $\eta\geq\theta^\kappa$, $\eta^\kappa=\eta$ and $2^{<\eta}\leq\eta^+$. Further suppose $\mathbb{E}$ is a poset such that 
\begin{center}
\begin{tabular}{rp{3.5in}}
$\forces^V_{\mathbb{E}}$ & there is a $(\kappa,\eta^+)$-extender $\dot{E}$ over the ground model $V$ with well-founded ultrapower $j_{\dot{E}}:{V}\to {M}$ such that $\crit(j_{\dot{E}})=\kappa$, $j_{\dot{E}}(\kappa)>\eta^+$, ${H}_{\eta^+}^V\subseteq {M}$, ${M}^{<\mu}\subseteq{M}$ and ${M}\models {H}^V_{\eta}\in \text{IA}_\mu$.
\end{tabular}
\end{center}
Then given any condition $p\in\Add(\kappa,\theta)^V$, the poset $\E$ also forces that there is an ultrafilter $g$ which is $(V,\Add(\kappa,\theta))$-generic with $p\in g$, and that the ultrapower map $j_{\dot{E}}$ can be lifted to $j^*_{\dot{E}}:V[g]\to M[j^*_{\dot{E}}(g)]$ where $H_{\eta^+}^{V[g]}\subseteq M[j^*_{\dot{E}}(g)]$, $M[j^*_{\dot{E}}(g)]^{<\mu}\subseteq M[j^*_{\dot{E}}(g)]$, and $M[j^*_{\dot{E}}(g)]\models H_\eta^{V[g]}\in \text{IA}_\mu$.
\end{theorem}

\begin{proof}
Let $\kappa\geq\omega_1^V$ and suppose $\theta$ and $\eta$ are cardinals with $\eta>\theta^+$, $\eta\geq\theta^\kappa$, $\eta^\kappa=\eta$ and $2^{<\eta}\leq\eta^+$. Suppose there is a partial order $\E$ such that if $G$ is $(V,\E)$-generic then there is a $V$-$(\kappa,\eta^+)$-extender $E\in V':=V[G]$ such that the ultrapower $j:V\to M=\Ult(V,E)\subseteq V'$ is well-founded and has the following properties.  

\begin{enumerate}
\item $\eta^+< j(\kappa)$ and the critical point of $j$ is $\kappa$.
\item $H_{\eta^+}^V\subseteq M$
\item $M^{<\mu}\subseteq M$
\item $M=\{j(f)(a)\st a\in[\eta^+]^{<\omega} \land f:[\kappa]^{|a|}\to V\land f\in V\}$
\item $M\models H^V_{\eta}\in \text{IA}_\mu$
\end{enumerate}
Let us define $\P=\Add(\kappa,\theta)^V$ and fix a condition $p\in\P$. Our goal is to show that in $V'$, there is a $(V,\P)$-generic filter $g$, with $p\in g$, such that the embedding $j$ can be lifted to domain $V[g]$ in such a way that the lifted embedding $j^*:V[g]\to M[j^*(g)]$ is a class of $V'$ and satisfies the following: $H^{V[g]}_{\eta^+}\subseteq M[j^*(g)]$, $M[j^*(g)]^{<\mu}\cap V'\subseteq M[j^*(g)]$, and $M[j^*(g)]\models H_\eta^{V[g]}\in \text{IA}_\mu$. To accomplish this we will build a $(V,\P)$-generic filter $g$ in $M$, an $(M,j(\P))$-generic filter $H$ in $V[G]$, and then we will modify $H$ to obtain another $(M,j(\P))$-generic filter $H^*$ with $j"g\subseteq H^*$.

\begin{lemma}\label{lemmajrestrictW}
If $W\in([H_{\eta^+}]^{\leq\kappa})^V$ then $j\restrict W\in M$.
\end{lemma}

\begin{proof}
Suppose $W\in ([H_{\eta^+}]^{\leq\kappa})^V$. First notice that $W\in M$ since $([H_{\eta^+}]^{\leq\kappa})^V\subseteq H_{\eta^+}^V\subseteq M$. Furthermore, $j[W]\in M$ because the elements of $W$ can be enumerated in a $\kappa$ sequence $\vec{W}$, and then $j[W]$ is obtained as the range of the sequence $j(\vec{W})\restrict\kappa\in M$. By elementarity $(W,\in)\cong (j[W],\in)$, and thus the Mostowski collapse of $(W,\in)$ equals that of $(j[W],\in)$. Thus, in $M$, there are isomorphisms 
\[(W,\in)\stackrel{\pi_0}{\longrightarrow}(\bar{W},\in)\stackrel{\pi_1}{\longrightarrow}(j[W],\in)\]
where $\pi_0$ is the Mostowski collapse of $W$ and $\pi_1$ is the inverse of the Mostowski collapse of $j[W]$. It follows that $j\restrict W=\pi_1\circ\pi_0\in M$.
\end{proof}


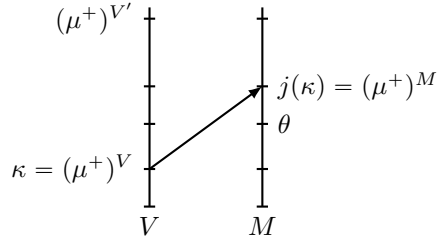
\begin{center}
\begin{tikzpicture}[scale=0.5,>=latex]
\draw[thick] (0,0) -- (0,5.3);
\draw[thick] (3,0) -- (3,5.3); 
\draw[->,thick] (0,1) -- (3,3.2) [];

\draw[thick] (0,0) node [anchor=north] {$V$};
\draw[thick] (-0.15,0) -- (0.15,0);
\draw[thick] (-0.15,1) node [anchor=east] {$\kappa=(\mu^+)^V$} -- (0.15,1);

\draw[thick] (-0.15,2.2) -- (0.15,2.2);

\draw[thick] (-0.15,3.2) -- (0.15,3.2);
\draw[thick] (-0.15,5) node [anchor=east] {$(\mu^+)^{V'}$} -- (0.15,5);

\draw[thick] (3,0) node [anchor=north] {$M$};
\draw[thick] (2.85,0) -- (3.15,0);
\draw[thick] (2.85,1) -- (3.15,1);
\draw[thick] (2.85,2.2) -- (3.15, 2.2) node [anchor=west] {$\theta$};
\draw[thick] (2.85,3.2) -- (3.15,3.2) node [anchor=west] {$j(\kappa)=(\mu^+)^M$};
\draw[thick] (2.85,5) -- (3.15,5);
\end{tikzpicture}
\end{center}

\begin{claim}\label{claim_Vgeneric}
There is a $(V,\Add(\kappa,\theta))$-generic filter $g\in M$ such that $p\in g$.
\end{claim}

\begin{proof}[Proof of Claim \ref{claim_Vgeneric}]

Working in $M$, we would like to apply Lemma \ref{lem_IA_gives_generics} to the structure $H=H^V_{\eta}$ and the poset $\Q=\Add(\kappa,\theta)\in H$. We will show that our assumption $M\models H^V_{\eta}\in\text{IA}_\mu$ implies that all of the other hypotheses of Lemma \ref{lem_IA_gives_generics} hold. First let us show that, working in $M$, for $N\in H$, if $|N|^M<\mu$ then $|N|^H<\mu$. Assume not, then for some $N\in H$ we have $|N|^M<\mu$ and $|N|^V=|N|^H \geq \mu$. Let $z\in P(N)^V$ with $|z|^V=\mu$, then by elementarily, $|j(z)|^{M}=\mu$. From Lemma \ref{lemmajrestrictW}, it follows that $j\restrict z\in M$ and since $|z|^V<\crit(j)=\kappa$ we have $j(z)=j"z\in M$. Thus $M\models$ $\mu=|j(z)|=|j"z|=|z|\leq|N|$, a contradiction. It is easy to verify that all of the other hypotheses of Lemma \ref{lem_IA_gives_generics} are met, and hence we may apply Lemma \ref{lem_IA_gives_generics} within $M$ to obtain a $(H,\Add(\kappa,\theta))$-generic filter $g$ with $p\in g$. Since $H=H^V_{\eta}$, it follows that $g$ is also $(V,\Add(\kappa,\theta))$-generic.
\end{proof}

\begin{claim}\label{claim_Mgeneric}
There is an $(M,j(\Add(\kappa,\theta)))$-generic filter $h\in V'=V[G]$. 
\end{claim}

\begin{proof}
All of $M$'s dense subsets of $j(\Add(\kappa,\theta))$ are elements of $j(H^V_\eta)$, so it suffices to prove that, in $V[G]$, the model $j(H^V_\eta)$ and the poset $j(\Add(\kappa,\theta)) \in j(H^V_\eta)$ satisfy all the requirements of Lemma \ref{lem_IA_gives_generics}, which will yield the desired generic object $h$.  The proof uses  the somewhat ad hoc requirement (\ref{item_ClosedLessMu}) of Definition \ref{def_GenericStrongIA} (though this requirement is redundant when $\kappa = \omega_1$, by Remark \ref{rem_RedundantAtOmega1}).  We do not know if this requirement can be removed, but suspect it is possible.  

Now clearly $j(H^V_\eta)$ believes that $j(\Add(\kappa,\theta))$ is ${<}j(\kappa)$-closed, and so in particular ${<}\mu$-closed, i.e., ${<}j(\mu)$-closed. Also, notice that $V[G]\models j(H_\eta)^V\in\text{IA}_\mu$ because $M\models H_\eta^V\in\text{IA}_\mu$ and if $\vec{N}=\<N_i\st i<\mu\>\in M$ witnesses this, then $\vec{N}\restrict\ell\in H_\eta^V\in V$ for each $\ell<\mu$, and hence $\<j(N_i)\st i<\mu\>\in V[G]$. To show that $j(H^V_\eta)$ and $j(\Add(\kappa,\theta))$ satisfy the remaining requirements of Lemma \ref{lem_IA_gives_generics} (from the point of view of $V[G]$), it suffices to prove:
\begin{enumerate}
 \item\label{item_ImageSmallInVG} $V[G] \models |j(H^V_\eta)| \le \mu$; and
 \item\label{item_ImageClosed} $V[G] \models$ $j(H^V_\eta)$ is closed under ${<}\mu$-sequences.
\end{enumerate}
To see item (\ref{item_ImageSmallInVG}):  since 
\begin{equation*}
M=\{j(f)(a)\st a\in[\eta^+]^{<\omega} \land f:[{\kappa}]^{|a|}\to V\land f\in V\}
\end{equation*}
it follows that every element of $j(H^V_\eta)$ is of the form $j(f)(a)$ where $a\in[\eta^+]^{<\omega}$ and $f:[{\kappa}]^{|a|}\to H^V_\eta$ is some function in $V$. Thus, $|j(H^V_\eta)|^{V[G]} \leq (\eta^+)^V \cdot | ^{\kappa}(H^V_\eta)|^V=\eta^{+V}$; the last equality follows from our assumption that $(2^{<\eta})^V=(\eta^+)^V$.  Since $\eta^{+V} < j(\kappa)$ by assumption, $\eta^{+V} < j(\kappa) = j(\mu^+) = \mu^{+M} \le \mu^{+V[G]}$.  This completes the proof of item (\ref{item_ImageSmallInVG}).

To see item (\ref{item_ImageClosed}):  $V$ believes that $H^V_\eta$ is closed under ${<}\mu$ sequences, so by elementarity $M$ believes that $j(H^V_\eta)$ is closed under ${<}\mu$ sequences.  The ad hoc assumption (\ref{item_ClosedLessMu}) of Definition \ref{def_GenericStrongIA} says that $M$ is closed under ${<}\mu$ sequences in $V[G]$.  Thus item (\ref{item_ImageClosed}) trivially follows.
\end{proof}

We will modify the $(M,j(\P))$-generic $h$ given by Claim \ref{claim_Mgeneric}, to obtain $h^*$ such that $h^*$ is still $(M,j(\P))$-generic and $j[g]\subseteq h^*$.

Working in $V[G]$, we may define a function $Q:=\bigcup j[g]$ where $\dom(Q)=j[\theta]\times\kappa$. Given $p\in h\subseteq \Add(j(\kappa),j(\theta))^M$ we define $p^*$ to be the function with $\dom(p^*)=\dom(p)$ such that $p^*(\alpha,\beta)=p(\alpha,\beta)$ if $(\alpha,\beta)\in\dom(p)-\dom(Q)$ and $p^*(\alpha,\beta)=Q(\alpha,\beta)$ if $(\alpha,\beta)\in\dom(Q)\cap\dom(p)$. Observe that the subset of $\dom(p)$ on which modifications are made is contained in the range of $j$:
\[\Delta_p:=\{(\alpha,\beta)\in\dom(p)\st p^*(\alpha,\beta)\neq p(\alpha,\beta)\}\subseteq j[\theta]\times\kappa\subseteq\ran(j).
\]
Now we let $h^*:=\{p^*\st p\in h\}\in V[G]$ and argue that $h^*$ is an $(M,j(\P))$-generic filter.

\begin{lemma}
If $p\in h$ then $h^*\in j(\P)$.
\end{lemma}

\begin{proof}
It follows that $p=j(f)(a)$ where $a\in[\eta^+]^{<\omega}$ and $f:[\kappa]^{|a|}\to \P$ is a function in $V$. Let $W:=\bigcup\{\dom(q)\st q\in\ran(f)\}\subseteq\theta\times\kappa$. Clearly $W\in([H_\theta]^{\leq\kappa})^V$ and thus $j\restrict W\in M$ by Lemma \ref{lemmajrestrictW}. Furthermore, we have $\dom(p)\subseteq j(W)$, which implies $\dom(p)\cap\ran(j)\subseteq j[W]$. Observe that any modifications made in obtaining $p^*$ from $p$ must have been made over points in $j[W]$; in other words, $\Delta_p \subseteq j[W]$.
Working in $M$, we can use $j\restrict W$ and $g$ to define a function $Q_p$ with $\dom(Q_p)=j[W]$ as follows. Define $Q_p(\alpha,\beta)=i$ if and only if $g(j^{-1}(\alpha,\beta))=i$. It follows that $Q_p=Q\restrict j[W]\in M$. Since $p^*$ can be obtained by comparing $p$ and $Q_p$, it follows that $p^*$ is in $M$, and is thus a condition in $j(\P)$.
\end{proof}

\begin{lemma}
$h^*$ is an $(M,j(\P))$-generic filter.
\end{lemma}
\begin{proof}
Suppose $A$ is a maximal antichain of $j(\P)$ in $M$ and let $d:=\bigcup\{\dom(r)\st r\in A\}$. We have $|d|^M\leq j(\kappa)$ by elementarity since $\P$ is $(2^{<\kappa})^+$-c.c. in $V$ and $2^{<\kappa}=\kappa$ by hypothesis. Since $d\in M$ there is an $a\in[\eta^+]^{<\omega}$ and a function $f:\kappa\to[\theta\times\kappa]^{\leq\kappa}$ in $V$ such that $d=j(f)(a)$. It follows that $W:=\bigcup \ran(f)\in ([H_\theta]^{\leq\kappa})^V$ and hence $j\restrict W\in M$ by Lemma \ref{lemmajrestrictW}. Since $d\in \ran(j(f))$ it follows that $d\subseteq j(W)$ and hence $d\cap\ran(j)\subseteq j[W]$.  We also have $j[W]\in M$. Working in $M$, we can define a function $Q_A$ with $\dom(Q_A)=j[W]$ by letting $Q_A(\alpha,\beta)=i$ if and only if $g(j^{-1}(\alpha,\beta))=i$. Notice that $Q_A=Q\restrict j[W]$ and $|j[W]|^M\leq\kappa$. By the ${<}j(\kappa)$-distributivity of $j(\P)$ in $M$ we have that $h\restrict j[W]\in M$ and thus $h\restrict j[W]\in j(\P)$. Furthermore, since $Q_A\in M$, there is an automorphism $\pi_A\in M$ of $j(\P)$ which flips precisely those bits at which $Q_A$ and $h\restrict j[W]$ disagree. Since $\pi_A^{-1}[A]\in M$ is a maximal antichain of $j(\P)$ and $h$ is $(M,j(\P))$-generic, it follows that there is a condition $q\in \pi_A^{-1}[A]\cap h$ and by applying $\pi_A$ we see that $q^*=\pi_A(q)\in A\cap h^*$.
\end{proof}

Since $j[g]\subseteq h^*$, it follows that the embedding $j:V\to M$ lifts to $j^*:V[g]\to M[j^*(g)]$ where $j^*(g)=h^*$. It remains to argue that $H^{V[g]}_{\eta^+}\subseteq M[j^*(g)]$, $M[j^*(g)]^{<\mu}\cap V'\subseteq M[j^*(g)]$, and $M[j^*(g)]\models H_\eta^{V[g]}\in \text{IA}_\mu$.

To show that $H^{V[g]}_{\eta^+}\subseteq M[j^*(g)]$ we note that because $g\in M$ and $H_{\eta^+}^V\subseteq M$ we have $H_{\eta^+}^{V[g]}=H^V_{\eta^+}[g]\subseteq M\subseteq M[j^*(g)]$. One can easily verify that $M[j^*(g)]^{<\mu}\cap V'\subseteq M[j^*(g)]$ using the fact that $j(\P)=\Add(j(\kappa),j(\theta))^M$ is ${<}j(\kappa)$-distributive in $M$. 

It remains to show that $M[j^*(g)]\models H_\eta^{V[g]}\in \text{IA}_\mu$. Since we assumed $M\models H_\eta^V\in \text{IA}_\mu$, there is a $\subseteq$-increasing sequence $\vec{N}=\<N_i\st i<\mu\>\in M$ satisfying all the requirements of Definition \ref{defn_mu_internally_approachable}: 
\begin{itemize}
 \item $N_\ell=\bigcup_{i<\ell}N_i$ for each limit $\ell<\mu$,
 \item $|N_i|^M < \mu$ for each $i < \mu$,
 \item $H_\eta^V = \bigcup_{i < \mu} N_i$, and
 \item $\langle N_i \ : \ i < \ell \rangle \in H_\eta^V$ for all $\ell < \mu$.
\end{itemize}
Furthermore, without loss of generality we can assume that
\begin{itemize}
\item $\P=\Add(\kappa,\theta)^V\in N_0$ and
\item $N_i\elemsub H_\eta^V$.
\end{itemize}
We now argue that the sequence $\<N_i[g]\st i<\mu\>\in M[j^*(g)]$ witnesses that $M[j^*(g)]\models H_\eta^{V[g]}\in \text{IA}_\mu$ where \[N_i[g]=\{\tau_{g} \st \text{$\tau\in N_i$ is a $\P$-name}\};\]
that is, we will prove the following.
\begin{enumerate}
 \item $N_\ell[g]=\bigcup_{i<\ell}N_i[g]$ for each limit $\ell < \mu$,
 \item $|N_i[g]|^{M[j^*(g)]} < \mu$ for each $i < \mu$,
 \item $H_\eta^{V[g]} = \bigcup_{i < \mu} (N_i[g])$, and
 \item $\langle N_i[g] \ : \ i < \ell \rangle \in H_\eta^{V[g]}$ for all $\ell < \mu$.
\end{enumerate}

Item $(1)$ follows directly from the definition of $N_i[g]$ and the fact that $\vec{N}$ is $\subseteq$-continuous. For $(2)$, we have $|N_i|^M<\mu$, which implies that $|N_i[g]|^{M[j^*(g)]}<\mu$.  For (3), $H_\eta^{V[g]}=H_\eta^V[g]=\bigcup_{i<\mu}(N_i[g])$. To establish $(4)$, we remark that by assumption we have $\<N_i\st i<\ell\>\in H_\eta^V$ for each $\ell<\mu$, and since $g\in H_\eta^{V[g]}=H_\eta^V[g]$, it follows that we can build $\<N_i[g]\st i<\ell\>$ working in $H_\eta^{V[g]}$.


This concludes the proof of Theorem \ref{theoremsurgery}.
\end{proof}

Theorem \ref{theoremsurgery} already implies the portion of Theorem \ref{maintheorem_generalized} which deals with preservation of generically strong cardinals:
\begin{corollary}\label{cor_PreserveGenericStrong}
If $\kappa= \mu^+$ is generically strong on $\text{IA}_\mu$ then it remains generically strong on $\text{IA}_\mu$ after adding any number of Cohen subsets of $\kappa$.
\end{corollary}
The next section combines Theorem \ref{theoremsurgery} with a generalization of Foreman's Duality Theorem to get preservation of ideally strong cardinals.

\section{Ideal extenders and duality}\label{sec_Duality}

This section generalizes portions of the ``duality theory" of Foreman~\cite{Foreman:CalculatingQuotientAlgebras} to the context of extender embeddings.

\subsection{Derived ideal extenders} 

The following definition and lemma are simple generalizations of the discussion in Section 3.2 of \cite{Foreman:Handbook}.
\begin{definition}\label{def_DerivedIdealExtender}
Suppose $\Q$ is a poset, $\kappa< \lambda$ are ordinals, and $\dot{E}$ is a $\Q$-name such that 
\begin{equation*}
\Vdash_{\Q} \dot{E} \text{ is a normal } (\kappa, \lambda) \text{-extender over } V
\end{equation*}
Then $F(\dot{E})$, called the \emph{ideal extender derived from $\dot{E}$}, is defined by: $(a,S) \in F(\dot{E})$ iff:
\begin{itemize}
 \item $a \in [\lambda]^{<\omega}$;
 \item $S \subseteq [\kappa]^{|a|}$; and 
 \item $\llbracket (\check{a}, \check{S}) \in \dot{E} \rrbracket_{\text{ro}(\Q)} = 1$.
\end{itemize}
The name ``ideal extender" will be justified in Lemma \ref{lem_IdealExtenderDerived} below.
\end{definition}

Note we do \emph{not} require that $\Q$ forces $\text{Ult}(V,\dot{E})$ to be well-founded in Definition \ref{def_DerivedIdealExtender}.

\begin{lemma}\label{lem_IdealExtenderDerived}
Suppose $\Q$ is a poset, $\kappa<\lambda$ are ordinals, and $\dot{E}$ is a $\Q$-name such that $\forces_\Q$ $\dot{E}$ is a normal $(\kappa,\lambda)$-extender over $V$. Let $F(\dot{E})$ be the ideal extender derived from $\dot{E}$ as in Definition \ref{def_DerivedIdealExtender}.  Then:
\begin{enumerate}
 \item\label{item_IndeedAnIdealExtender} $F(\dot{E})$ is a $(\kappa, \lambda)$-ideal extender in the sense of Definition \ref{def_IdealExtender}; let $\mathbb{E}_{F(\dot{E})}$ be the associated forcing as in Definition \ref{def_AssocForcing}.
 \item\label{item_BooleanEmbedding} The map
\begin{equation*}
\iota: \mathbb{E}_{F(\dot{E})} \to \ro(\mathbb{Q})
\end{equation*}
defined by
\begin{equation*}
p \mapsto \llbracket \check{F}^p \subset \dot{E}  \rrbracket_{\text{ro}(\mathbb{Q})} = \llbracket  \forall a \in \supp(\check{p})  \ \forall S \in \check{F}^p_a \ S \in \dot{E}_a  \rrbracket_{\mathrm{ro}(\mathbb{Q})}
\end{equation*} 
 is $\le$ and $\perp$ preserving (but not necessarily regular).  
 \item\label{item_GenericTransfer} If $H$ is $(V,\mathbb{Q})$-generic then $\bigcup\iota^{-1}[H] = \dot{E}_H$.
\end{enumerate}
\end{lemma}
\begin{proof}
(\ref{item_IndeedAnIdealExtender}) The fact that $F(\dot{E})$ is a $(\kappa,\lambda)$-ideal extender, as in Definition \ref{def_IdealExtender}, follows from the fact that the relevant properties hold for $\dot{E}$ with boolean-value one in $V^{\ro(\Q)}$.  It is easy to see that for each $a\in[\lambda]^{<\omega}$ the set $F(\dot{E})_a\subsetneq \wp([\kappa]^{|a|})$ is a filter since $\llbracket\dot{E}_a\textrm{ is an ultrafilter over $V$}\rrbracket_\Q=1$. To see that $F(\dot{E})$ is compatible we observe that the coherence property of $(\kappa,\lambda)$-extenders \cite[(ii) on page 154]{Kanamori:Book} holds for $\dot{E}$ with boolean value one in $V^{\ro(\Q)}$, and this immediately implies that if $a\subseteq b\in\supp(F)$ then $x\in F(\dot{E})_a \iff x_{a,b}\in F(\dot{E})_b$.  The potential normality of $F(\dot{E})$ follows from the assumption that $1_{\mathbb{Q}}$ forces $\dot{E}$ to be a normal extender; we refer the reader to the proof at the bottom of page 66 of Claverie~\cite{Claverie:IdealsIdealExtendersAndForcingAxioms} for the details.  Thus $F(\dot{E})$ is a $(\kappa,\lambda)$-ideal extender, as in Definition \ref{def_IdealExtender}.

(\ref{item_BooleanEmbedding}) Suppose $p\leq_{\mathbb{E}_{F(\dot{E})}} q$. This means that $\supp(q)\subseteq\supp(p)$ and for every $a\in\supp(q)$ we have $F^q_a\subseteq F^p_a$. Thus for $a\in\supp(q)$ we have $\llbracket\check{F}^p_a\subseteq\dot{E}_a\rrbracket\leq_{\ro(\Q)}\llbracket\check{F}^q_a\subseteq\dot{E}_a\rrbracket$. Thus
$\iota(p)=\bigwedge_{a\in\supp(p)}\llbracket\check{F}^p_a\subseteq\dot{E}_a\rrbracket\leq_{\ro(\Q)}\bigwedge_{a\in\supp(q)}\llbracket\check{F}^q_a\subseteq\dot{E}_a\rrbracket=\iota(q)$.

Suppose $p,q\in\mathbb{E}_{F(\dot{E})}$ where $p=\span\{F(\dot{E}),$ $(a,x)\}$ and $q=\span\{F(\dot{E}),$ $(b,y)\}$ and that $\iota(p)$ and $\iota(q)$ are compatible in $\ro(\Q)$, say $r\in\Q$ with $r\leq\iota(p),\iota(q)$. Then forcing with $\Q$ below $r$ yields an extension in which there is a $(\kappa,\lambda)$-extender $F$ such that $F^p\cup F^q\subseteq F$. Since $x\in F_a$ and $y\in F_b$, it follows that $x_{a,a\cup b}\cap y_{b,a\cup b}\in F_{a\cup b}$ and hence $x_{a,a\cup b}\cap y_{b,a\cup b}\in F(\dot{E})_{a\cup b}^+$. This implies that \[s:=\span\{F(\dot{E}),(a\cup b,x_{a,a\cup b}\cap y_{b,a\cup b})\}\] is a condition in $\mathbb{E}_{F(\dot{E})}$. Since $x_{a,a\cup b}\in F^s_{a\cup b}$ and $s$ is a compatible $(\kappa,\lambda)$-system of filters, it follows that $x\in F^s_a$. Similarly, $y\in F^s_b$. This implies that $s\leq p,q$.

(\ref{item_GenericTransfer}) To see that $\bigcup\iota^{-1}(H)=\dot{E}_H$, just note that $\bigcup\iota^{-1}[H]=\bigcup\{p\in\E_{F(\dot{E})}:\iota(p)\in H\}=\bigcup\{p\in\E_{F(\dot{E})}:F^p\subseteq \dot{E}_H\}=\dot{E}_H$. \end{proof}

\subsection{Duality for derived ideal extenders}

In certain situations, the map $\iota$ from Lemma \ref{lem_IdealExtenderDerived} is a regular or even a dense embedding.  In such situations we can nicely characterize the poset $\mathbb{E}_{F(\dot{E})}$ as a subalgebra of $\mathbb{Q}$; and moreover if $\Ult(V,\dot{E})$ was forced by $\mathbb{Q}$ to be well-founded, then $F(\dot{E})$ will be a \emph{precipitous} ideal extender.  

In the following theorem, the poset which plays the role of the $\mathbb{Q}$ from Definition \ref{def_DerivedIdealExtender} and Lemma \ref{lem_IdealExtenderDerived} above is itself a quotient of a forcing associated with an ideal extender.

\begin{theorem}\label{thm_Duality_Extenders}

Suppose $F$ is a $(\kappa, \lambda)$-ideal extender as in Definition \ref{def_IdealExtender} and that $\mathbb{E}_{F}$ is its associated forcing as in Definition \ref{def_AssocForcing}.     Suppose $\mathbb{P}$ is a poset and $\tau$ is an $\mathbb{E}_F$-name such that $\mathbb{E}_F$ forces the following statement (here $\dot{G}_{F}$ is the $\mathbb{E}_{F}$-name for the $\mathbb{E}_{F}$-generic object; so $\bigcup\dot{G}_F$ is forced to be a $(\kappa, \lambda)$-extender over $V$):

\begin{itemize}
 \item  $\tau$ is $(V,\mathbb{P})$-generic, and the ultrapower map $j_{\dot{G}_{F}}$ can be lifted to domain $V[\tau]$.\footnote{We're assuming this lifting exists already in $V^{\mathbb{E}_F}$, not in some further forcing extension.} 
\end{itemize}

Then:
\begin{enumerate}
 \item\label{item_AbstractFact} The map $e: \mathbb{P} \to \ro(\mathbb{E}_{F})$ defined by $e(p)=\llbracket\check{p}\in\tau\rrbracket$ is a complete embedding (this is an abstract forcing fact).
 \item\label{item_IotaIsDense} Suppose $g$ is $(V,\mathbb{P})$-generic. In $V[g]$ let:
 \begin{itemize} 
  \item $\tilde{j}$ be the $\frac{\mathbb{E}_F}{e[g]}$-name for the lifting of $j_{{\dot{E}}^{\dot{G}}_F}$ to domain $V[g]$ (notice that $\forces_{\frac{\mathbb{E}_F}{e[g]}}\tau_{\dot{G}_F}=g$ so the lifting is forced to exist by assumption); 
  \item $\dot{E}$ be the $\frac{\mathbb{E}_F}{e[g]}$-name for the  $(\kappa, \lambda)$-extender over $V[g]$ derived from $\tilde{j}$;
  \item $F(\dot{E})$ be the ideal extender derived from the $\frac{\mathbb{E}_F}{e[g]}$-name $\dot{E}$, as in Definition \ref{def_DerivedIdealExtender}. 
  \item $\mathbb{E}_{F(\dot{E})}$ be the forcing in $V[g]$ associated with $F(\dot{E})$ as in Definition \ref{def_AssocForcing}.
 \end{itemize}  
 
 Then in $V[g]$ the map
 \begin{equation*}
\iota:  \mathbb{E}_{F(\dot{E})} \to  \ro\left(\frac{\mathbb{E}_{F}}{e[g]}\right)
 \end{equation*}
 defined as in Lemma \ref{lem_IdealExtenderDerived} is a dense embedding.
   \item\label{item_Precipitous}  If the $(\kappa,\lambda)$-ideal extender $F$ was precipitous in $V$ then $F(\dot{E})$ is precipitous in $V[g]$.
\end{enumerate}
\end{theorem}
\begin{proof}
(\ref{item_AbstractFact}) is a well-known fact.

(\ref{item_IotaIsDense})  A minor technical issue here is that although $\mathbb{E}_F$ is separative by Lemma \ref{lem_E_F_separative}---and thus $\mathbb{E}_F$ can directly be viewed as a dense subset of $\ro (\mathbb{E}_F)$---the quotient $\Q:= \frac{\mathbb{E}_F}{e[g]}$ is not separative; so $\Q$ is technically not a subset of $\ro(\Q)$.  Let $\Q/{\sim}$ denote the separative quotient of $\Q$.  Since $\Q /{\sim}$ is a dense subset of $\ro(\Q)$, it suffices to show that given any $[q]_\sim \in \Q/{\sim}$ there is an $r\in\mathbb{E}_{F(\dot{E})}$ with $\iota(r)\leq [q]_\sim$. Suppose $[q]_\sim \in\Q /{\sim}$. Then $q=F^q=\span\{F,(a,x)\}$ for some fixed $a\in[\lambda]^{<\omega}$ and some $x\in (F_a)^+$, and $q$ is compatible with every element of $e[g]$.\footnote{Recall that by Lemma \ref{lem_E_F_separative}, we know $\mathbb{E}_F$ is separative and thus can be viewed as a dense subset of $\ro(\mathbb{E}_F)$, the target of the map $e$.} Notice that $x\in(F(\dot{E})_a)^+$ since (a) $q$ is compatible with every element of $e[g]$, (b) $0\neq q\forces x\in\dot{E}$ and (c) the ideal dual to the filter $F(\dot{E})_a=\{S\subseteq\kappa:\llbracket (\check{a},\check{S})\in\dot{E}\rrbracket_{\ro(\Q)}=1\}$ is $\{S\subseteq\kappa:\llbracket(\check{a},\check{S})\in\dot{E}\rrbracket=0\}$. Thus $r:=\span\{F(\dot{E}),(a,x)\}$ is a condition in $\mathbb{E}_{F(\dot{E})}$ and by definition of $\iota$ we have $\iota(r)=\llbracket\forall b\in\supp(r)\ \check{F}^r_b\subseteq\dot{E}_b\rrbracket_{\ro(\Q)}$. We will argue that $\iota(r)\leq [q]_\sim$. Notice that since $F\subseteq F(\dot{E})$, it follows that $q=\span\{F,(a,x)\}\subseteq \span\{F(\dot{E}),(a,x)\}=r$. Note that $q$ and $r$ are conditions in different posets, so this does not imply that $r$ extends $q$. However, since for all $b\in\supp(q)$ we have $F^q_b\subseteq F^r_b$, it follows that if $H$ is $\left(V[g],\ro(\Q)\right)$-generic and $\iota(r)\in H$ then $[q]_\sim \in H$. Since $\ro(\Q)$ is separative, it follows that $\iota(r)\leq [q]_\sim$ in $\ro(\Q)$.

(\ref{item_Precipitous}) Assume the $(\kappa,\lambda)$-ideal extender $F$ is precipitous in $V$.  Let $g$ be $\big( V, \mathbb{P})$-generic.  By part \ref{item_GenericTransfer} of Lemma \ref{lem_IdealExtenderDerived}, together with part (\ref{item_IotaIsDense}) of the current theorem, generic ultrapowers of $V[g]$ by $\mathbb{E}_{F(\dot{E})}$ correspond exactly to liftings of generic ultrapowers of $V$ by $F$; since these are wellfouned (by precipitousness of $F$) the generic ultrapowers of $V[g]$ by $\mathbb{E}_{F(\dot{E})}$ are also wellfounded.
\end{proof}

\section{Proof of main theorem}\label{sectionproof}

In this section we combine the results of Section \ref{sec_Surgery} and Section \ref{sec_Duality} to finish proving Theorem \ref{maintheorem_generalized}.  Note that we only have to deal with preservation of ideally strong cardinals; preservation of generically strong cardinals was taken care of by Corollary \ref{cor_PreserveGenericStrong}.

\begin{proof}[Proof of Theorem \ref{maintheorem_generalized}]

Suppose $\mu$ is a regular cardinal, $\kappa=\mu^+$ is ideally strong on $\text{IA}_\mu$, and $2^{<\kappa}=\kappa$. Let $\theta>\kappa$ be a cardinal and let $g'$ be $(V,\Add(\kappa,\theta))$-generic. Suppose $\kappa$ is not ideally strong on $\text{IA}_\mu$ in $V[g']$. Thus, in $V[g']$, for some cardinal $\eta>\kappa$ meeting all the hypotheses of Theorem \ref{theoremsurgery}, there is no precipitous $(\kappa,\eta^+)$-ideal extender $F$ such that the associated generic ultrapower $j_G:V[g']\to M=\Ult(V[g'],\dot{E}^G_F)$ satisfies $j(\kappa)>\eta^+$, $H_{\eta^+}^V\subseteq M$, $M^{<\mu}\cap V[g']\subseteq M$,  $M\models H_\eta^{V[g']}\in\text{IA}_\mu$. Let $p_0 \in g'$ force this; that is, 
\begin{center}
\begin{tabular}{rp{3.5in}}
$p_0\forces^V_{\Add(\kappa,\theta)}$ & there is no precipitous $(\kappa,\eta^+)$-ideal extender $F$ such that the associated generic ultrapower $j_{\dot{G}_F}:{V[\dot{g}]}\to {M=\Ult(V[\dot{g}],\dot{G}_F)}$ satisfies $\crit(j_{\dot{G}_F})=\kappa$, $j_{\dot{G}_F}(\kappa)>\eta^+$, ${H}_{\eta^+}^{V[\dot{g}]}\subseteq {M}$, ${M}^{<\mu}\cap V[\dot{g}]\subseteq{M}$ and ${M}\models {H}^{V[\dot{g}]}_{\eta}\in \text{IA}_\mu$.
\end{tabular}
\end{center}
where $\dot{g}\in V$ is the canonical name for a generic filter for $\Add(\kappa,\theta)$. Using the fact that $\kappa$ is ideally strong on $\text{IA}_\mu$ in $V$, let $F\in V$ be a precipitous $(\kappa,\eta^+)$-ideal extender such that the associated generic ultrapower $j_{G_F}:V\to M=\Ult(V,G_F)$ satisfies $\crit(j_{G_F})=\kappa$, $j_{G_F}(\kappa)>\eta^+$, $H^V_{\eta^+}\subseteq M$, $M^{<\mu}\cap V\subseteq M$ and $M\models H_\eta^V\in\text{IA}_\mu$, where $G_F$ is a $(V,\mathbb{E}_F)$-generic filter. By Theorem \ref{theoremsurgery}, it follows that we can find a $(V,\Add(\kappa,\theta))$-generic filter $g$ with $p_0 \in g$ and $g\in M$, such that the embedding $j_{G_F}$ lifts to $j^*_{G_F}:V[g]\to M[j^*_{G_F}(g)]$ and satisfies $H_{\eta^+}^{V[g]}\subseteq M[j^*_{G_F}(g)]$, $M[j^*_{G_F}(g)]^{<\mu}\cap V[g]\subseteq M[j^*_{G_F}(g)]$ and $M[j^*_{G_F}(g)]\models H_\eta^{V[g]}\in \text{IA}_\mu$, where $j^*_{G_F}(g)\in V[G_F]$ is some $(M,j(\Add(\kappa,\theta))$-generic filter. Thus there is an $\E_F$-name $\tau$ with $\tau_{G_F}=g$ satisfying the hypothesis of Theorem \ref{thm_Duality_Extenders} with $\P=\Add(\kappa,\theta)^V$. Not only does the hypothesis of Theorem \ref{thm_Duality_Extenders} hold, but Theorem \ref{theoremsurgery} also shows that 
\begin{equation}\label{eqn_forced_strength}
\textrm{
\begin{tabular}{rp{3.5in}}
$\forces_{\E_F}^V$ & if the original ultrapower map $j_{\dot G_F}:V\to M=\Ult(V,\dot G_F)$ satisfies $\crit(j_{\dot{G}_F})=\kappa$, $j_{\dot{G}_F}(\kappa)>\eta^+$, $H_{\eta^+}^V\subseteq M$, $M^{<\mu}\cap V\subseteq M$, and $M\models H_\eta^{V}\in \text{IA}_\mu$, then the embedding lifted to $j^*_{\dot G_F}:V[\tau]\to M[j^*_{\dot G_F}(\tau)]$ satisfies $H_{\eta^+}^{V[\tau]}\subseteq M[j^*_{\dot{G}_F}(\tau)]$, $M[j^*_{\dot{G}_F}(\tau)]^{<\mu}\cap V[\tau]\subseteq M[j^*_{\dot{G}_F}(\tau)]$ and $M[j^*_{\dot{G}_F}(\tau)]\models H_\eta^{V[\tau]}\in \text{IA}_\mu$.\\
\end{tabular}}
\end{equation}
Applying Theorem \ref{thm_Duality_Extenders}, the embedding $e:\P\to\ro(\E_F)$ defined by $e(p)=\llbracket \check{p}\in \tau \rrbracket$ is complete. As in the conclusion of Theorem \ref{thm_Duality_Extenders}, in $V[g]$ we have the following objects: $\tilde{j}$ is the $\E_F/e[g]$-name $j^*_{\dot G_F}/e[g]$ for the lift of $j_{\dot G_F}$, $\dot{E}$ is the $\E_F/e[g]$-name for the $(\kappa,\eta^+)$-extender over $V[g]$ derived from $\tilde{j}$, $F(\dot{E})$ is the $(\kappa,\eta^+)$-ideal extender in $V[g]$ derived from $\dot{E}$, $\E_{F(\dot{E})}$ is the associated forcing and $\iota:\E_{F(\dot{E})}\to\ro(\frac{\E_F}{e[g]})$ is the natural dense embedding. Now if $H_{F(\dot{E})}$ is $(V[g],\E_{F(\dot{E})})$-generic and $H_{F}=\iota[H_{F(\dot{E})}]$ is the corresponding $(V[g],\frac{\E_F}{e[g]})$-generic filter, then by Lemma \ref{lem_IdealExtenderDerived}, $\bigcup\iota^{-1}[H_F]=\dot{E}_{H_F}$ is the $V[g]$-$(\kappa,\eta^+)$-extender derived from $\tilde{j}_{H_F}$. Since $\iota^{-1}[H_F]=H_{F(\dot{E})}$ we have $\bigcup H_{F(\dot{E})}=\dot{E}_{H_F}$ and hence the ultrapower $j_{H_{F(\dot{E})}}:V[g]\to \bar{M}=\Ult(V[g],\bigcup H_{F(\dot{E})})$ by the $V[g]$-$(\kappa,\eta^+)$-extender $\bigcup H_{F(\dot{E})}$ is well-founded. Furthermore, by (\ref{eqn_forced_strength}), and the fact that $j_{H_{F(\dot{E})}}=\tilde{j}_{H_F}=(j^*_{\dot{G}_F})_{H_F}$, it follows that $H^{V[g]}_{\eta^+}\subseteq\Ult(V[g],\bigcup H_{F(\dot{E})})$. Thus, the precipitous $(\kappa,\eta^+)$-ideal extender $F(\dot{E})\in V[g]$ has an associated generic ultrapower embedding $j_{H_{F(\dot{E})}}:V[g]\to \bar{M}$ which satisfies $\crit(j_{H_{F(\dot{E})}})=\kappa$, $j_{H_{F(\dot{E})}}(\kappa)>\eta^+$, $H_{\eta^+}^{V[g]}\subseteq \bar{M}$, $\bar{M}^{<\mu}\cap V[g]\subseteq\bar{M}$ and $\bar{M}\models H_\eta^{V[g]}\in\text{IA}_\mu$. This is a contradiction because $p_0 \in g$ forces that such an ideal extender does not exist.
\end{proof}

\section{Questions}

\begin{question}
Aside from adding Cohen subsets of $\omega_1$, what other posets will not destroy the ideal or generic strongness of $\omega_1$?  Note that by the discussion in the introduction, only forcings which don't collapse $\omega_2$ have any hope of (provably) preserving generic strongness of $\omega_1$.
\end{question}

\begin{question}
Is the ideal strongness of $\omega_1$ indestructible by $\Sacks(\omega_1,\theta)$? 

\end{question}






\begin{question}
Does the theorem still hold for $\kappa\geq\omega_2$ if we remove requirement (\ref{item_ClosedLessMu}) from Definition \ref{def_GenericStrongIA} (i.e., the requirement that ${}^{<\mu}M \subset M$)? The only place where this assumption was used was the proof of Claim \ref{claim_Mgeneric} on page \pageref{claim_Mgeneric}.
\end{question}


\end{document}